\documentclass[10pt]{amsart}
\usepackage{amssymb, enumitem}
\usepackage[all]{xy}
\usepackage{hyperref, aliascnt}
\usepackage{mathtools}
\usepackage[english]{babel}

\def\today{\number\day\space\ifcase\month\or   January\or February\or
   March\or April\or May\or June\or   July\or August\or September\or
   October\or November\or December\fi\   \number\year}

\theoremstyle{definition}
\newtheorem{lma}{Lemma}[section]

\newaliascnt{thmCt}{lma}
\newtheorem{thm}[thmCt]{Theorem}
\aliascntresetthe{thmCt}

\newaliascnt{corCt}{lma}
\newtheorem{cor}[corCt]{Corollary}
\aliascntresetthe{corCt}

\newaliascnt{propCt}{lma}
\newtheorem{prop}[propCt]{Proposition}
\aliascntresetthe{propCt}

\newtheorem*{thm*}{Theorem}
\newtheorem*{qst*}{Question}
\newtheorem*{cor*}{Corollary}
\newtheorem*{prop*}{Proposition}
\newcounter{theoremintro}
\newtheorem{thmintro}[theoremintro]{Theorem}

\newtheorem{cnjintro}[theoremintro]{Conjecture}
\newaliascnt{pgrCt}{lma}

\aliascntresetthe{pgrCt}

\newaliascnt{dfCt}{lma}
\newtheorem{df}[dfCt]{Definition}
\aliascntresetthe{dfCt}

\newaliascnt{remCt}{lma}
\newtheorem{rem}[remCt]{Remark}
\aliascntresetthe{remCt}

\newaliascnt{remsCt}{lma}

\aliascntresetthe{remsCt}

\newaliascnt{egCt}{lma}

\aliascntresetthe{egCt}

\newaliascnt{egsCt}{lma}

\aliascntresetthe{egsCt}

\newaliascnt{qstCt}{lma}

\aliascntresetthe{qstCt}

\newaliascnt{pbmCt}{lma}
\newtheorem{pbm}[pbmCt]{Problem}
\aliascntresetthe{pbmCt}

\newaliascnt{notaCt}{lma}
\newtheorem{nota}[notaCt]{Notation}
\aliascntresetthe{notaCt}

\newcommand{\beq}{\begin{equation}}
\newcommand{\eeq}{\end{equation}}
\newcommand{\beqa}{\begin{eqnarray*}}
\newcommand{\eeqa}{\end{eqnarray*}}
\newcommand{\bal}{\begin{align*}}
\newcommand{\eal}{\end{align*}}
\newcommand{\bi}{\begin{itemize}}
\newcommand{\ei}{\end{itemize}}
\newcommand{\be}{\begin{enumerate}}
\newcommand{\ee}{\end{enumerate}}

\newcommand{\Ad}{{\mathrm{Ad}}}

\date{\today}

\title{Group amenability and actions on $\mathcal{Z}$-stable $C^*$-algebras}
\author[Eusebio Gardella]{Eusebio Gardella}
\address{Eusebio Gardella\\
Westf\"{a}lische Wilhelms-Universit\"{a}t M\"{u}nster, Fachbereich
Mathematik, Einsteinstrasse 62, 48149 M\"{u}nster, Germany}
\email{gardella@uni-muenster.de}
\urladdr{http://wwwmath.uni-muenster.de/u/gardella/}
\author[Martino Lupini]{Martino Lupini}
\address{Martino Lupini, Mathematics Department, California Institute of
Technology, 1200 East California Boulevard, Mail Code 253-37, Pasadena, CA
91125}
\email{lupini@caltech.edu}
\urladdr{http://www.lupini.org/}
\thanks{This work was initiated during a visit of the first named author
to the second at the California Institute of
Technology in January 2017, and was continued during a visit of both authors
to the Centre de Recerca Matem{\`{a}}tica in March 2017 in occasion of the
Intensive Research Programme on Operator Algebras. The authors gratefully
acknowledge the hospitality and the financial support of both institutions.
The first named author was partially funded by SFB 878
\emph{Groups, Geometry and Actions},
and by a postdoctoral fellowship from the Humboldt Foundation, and the second
named author was partially supported by the NSF Grant DMS-1600186.}
\dedicatory{}
\subjclass[2000]{Primary 46L55, 54H05; Secondary 03E15, 37A55}
\keywords{Amenable group, free group, Bernoulli shift, GNS construction, 
weak containment, cocycle equivalence, Jiang-Su algebra, hyperfinite II$_1$-factor}

\begin{document}

\begin{abstract}
We study strongly outer actions of discrete groups on C*-algebras in
relation to (non)amenability. 
In contrast to related results for amenable groups, where uniqueness of
strongly outer actions on the Jiang-Su algebra is expected, we show that
uniqueness fails for all nonamenable groups, and that the failure is
drastic. Our main result implies that if $G$ contains a copy of the free
group, then there exist uncountable many, non-cocycle conjugate strongly
outer actions of $G$ on any Jiang-Su stable tracial C*-algebra. Similar
conclusions apply for outer actions on McDuff tracial von Neumann algebras.
We moreover show that $G$ is amenable if and only if the Bernoulli shift on
a finite strongly self-absorbing C*-algebra absorbs the trivial action on
the Jiang-Su algebra. Our methods consist in a careful study of weak
containments of the Koopman representations of different Bernoulli-type
actions.
\end{abstract}

\maketitle


\renewcommand*{\thetheoremintro}{\Alph{theoremintro}}

\section*{Introduction}

Amenability for discrete groups was first introduced by von Neumann in the
context of the Banach-Tarski paradox. One of the main early results in the
theory, proved by Tarski, asserts that a group is amenable if and only if it
admits no paradoxical decompositions. The fact that the Banach-Tarski
paradox only makes use of free groups led Day to conjecture that a discrete
group is nonamenable if and only if it contains the free group $\mathbb{F}_2$
as a subgroup. This conjecture, known as \emph{the von Neumann problem}, was
open for about 40 years, until it was disproved by Ol'shanskii.

Amenability admits a surprisingly large number of equivalent formulations.
Here, we are concerned with those characterizations that are phrased in
terms of actions of the group. 
%
These usually come in the form of a dichotomy: roughly speaking, they assert
that there is an object in the relevant category, on which every amenable
group acts in an essentially unique way, while every nonamenable group
admits a continuum of non-equivalent actions. The following is an
illustrative example:

\begin{thm*}
Let $G$ be a discrete group, and let $(X,\mu)$ be a standard atomless
probability space.

\begin{enumerate}
\item If $G$ is amenable, then all free, measure preserving, ergodic actions
of $G$ on $(X,\mu)$ are orbit equivalent.

\item If $G$ is not amenable, then there exist uncountably many non-orbit
equivalent free, measure preserving, ergodic actions of $G$ on $(X,\mu)$.
\end{enumerate}
\end{thm*}

Part~(1) is a combination of classical results of Dye and Ornstein-Weiss. In
reference to (2), the first result in this direction is a theorem of
Connes-Weiss, asserting that every nonamenable group without property (T)
admits two such actions. Much more recently, Ioana proved part~(2) for
groups containing a copy of $\mathbb{F}_2$ (\cite{Ioa_orbit_2011}), using
the corresponding result for $\mathbb{F}_2$ due to Gaboriau-Popa (\cite%
{GabPop_uncountable_2005}), and finally Epstein extended the result to all
nonamenable groups (\cite{Eps_orbit_2007}). In recent work (\cite%
{GarLup_borel_2017}), the authors improved the conclusion in part~(2) above:
the relation of orbit equivalence of actions of nonamenable groups is not
Borel.

In the context of von Neumann algebras, and specifically for the hyperfinite
II$_1$-factor $\mathcal{R}$, amenability can also be characterized in terms
of actions:

\begin{thm*}
Let $G$ be a discrete group, and let $\mathcal{R}$ be the hyperfinite II$_1$%
-factor.

\begin{enumerate}
\item If $G$ is amenable, then all outer actions of $G$ on $\mathcal{R}$ are
cocycle conjugate.

\item If $G$ is not amenable, then there exist uncountably many non-cocycle
conjugate outer actions of $G$ on $\mathcal{R}$.
\end{enumerate}
\end{thm*}

Part~(1) is due to Ocneanu (\cite{Ocn_actions_1985}), although particular
cases were proved by Connes for cyclic groups (\cite{Con_periodic_1976}),
and by Jones for finite groups (\cite{Jon_finite_1980}). Part~(2) is a
recent result due to Brothier-Vaes (Theorem~B in~\cite{BroVae_families_2015}%
), generalizing previous results of Popa (\cite{Pop_some_2006}) and Jones (%
\cite{Jon_converse_1983}).

In both theorems recalled above, the amenable case was resolved relatively
early. On the other hand, the nonamenable case took much longer, and it
required the invention of new and powerful tools such as Popa's celebrated
deformation/rigidity theory. Indeed, it was realized that certain
nonamenable groups (or certain nonamenable II$_1$-factors) exhibit striking
rigidity phenomena, which is best seen in the presence of property (T). The
richness of the nonamenable world drove researchers in both Ergodic Theory
and in von Neumann algebras to study actions of nonamenable groups on the
standard atomless probability space as well as on $\mathcal{R}$,
particularly in what relates to the complexity of their classification.

This work revolves around analogs of the above results in the context of
C*-algebras, the central theme being the case of nonamenable groups.
Strongly self-absorbing C*-algebras can be seen as the C*-analog of the
hyperfinite II$_1$-factor. (Recall that a unital C*-algebra ${\mathcal{D}}$
is said to be \emph{strongly self-absorbing} if it is infinite dimensional
and there is an isomorphism $\varphi\colon {\mathcal{D}}\to {\mathcal{D}}%
\otimes_{\mathrm{min}}{\mathcal{D}}$ which is approximately unitarily
equivalent to the first tensor factor embedding; see \cite%
{TonWin_strongly_2007}.) Examples of such algebras are the UHF-algebras of
infinite type, the Jiang-Su algebra $\mathcal{Z}$, the Cuntz algebras $%
\mathcal{O}_2$ and $\mathcal{O}_\infty$, and their tensor products. (It is
believed that this list is complete.) By a result of Winter (\cite%
{Win_strongly_2011}), any strongly self-absorbing C*-algebra absorbs $%
\mathcal{Z}$ tensorially.

C*-analogs of part~(1) in the theorem above were explored in the early 1990s
by Bratteli, Evans and Kishimoto (\cite{BraEvaKis_rohlin_1995}), who studied
a concrete family of outer actions of ${\mathbb{Z}}$ on a specific
UHF-algebra. Their results show that outerness in (finite) C*-algebras is
too weak a condition for an analog of Ocneanu's result to hold. They also
provided evidence for the fact that a uniqueness result may hold if one
assumes outerness not only for the action, but also for its extension to the
weak closure in the GNS representation. This notion is now called \emph{%
strong outerness}.

Motivated by a recent breakthrough of Szab\'{o} \cite{Sza_stronglyIII_2016},
the following conjecture has been proposed in \cite{GarLup_cocycle_2016},
the first part of which had already been suggested in \cite%
{Sza_stronglyIII_2016}. We refer the reader to the introductions of \cite%
{Sza_stronglyIII_2016} and \cite{GarLup_cocycle_2016} for motivation and
relevant references (in particular, for the reason why torsion groups must
be excluded).

\begin{cnjintro}
\label{cnjintro} Let $G$ be a torsion-free countable group, and let $%
\mathcal{D}$ be a strongly self-absorbing C*-algebra.

\begin{enumerate}
\item If $G$ is amenable, then any two strongly outer actions of $G$ on $%
\mathcal{D}$ are cocycle conjugate.

\item If $G$ is not amenable, then there exist uncountably many non-cocycle
conjugate strongly outer actions of $G$ on $\mathcal{D}$.
\end{enumerate}
\end{cnjintro}

The main result of \cite{Sza_stronglyIII_2016} asserts that part~(1) holds
when ${\mathcal{D}}$ is either a UHF-algebra of the Jiang-Su algebra and $G$
is abelian, while \cite{GarLup_cocycle_2016} asserts that part~(2) holds
when ${\mathcal{D}}$ is a UHF-algebra and for groups containing a subgroup
with relative property (T).

In this work, we continue the study of strongly outer actions on
C*-algebras, and specifically on strongly self-absorbing C*-algebras. In
particular, we are interested in constructing many non-cocycle conjugate
actions for nonamenable groups. We focus on a specific and very rich class
of actions, which we call \emph{generalized (noncommutative) Bernoulli shifts%
}. These are constructed as follows: given a strongly self-absorbing
C*-algebra $\mathcal{D}$ and an action $G\curvearrowright^\sigma X$ of a
discrete group $G$ on a countable set $X$, we consider the action of $G$ on $%
\otimes_{x\in X}\mathcal{D}\cong\mathcal{D}$ given by permuting the tensor
factors according to $\sigma$.

For an arbitrary group $G$, it seems difficult to produce actions of this
form other than the usual Bernoulli shift $G\curvearrowright \otimes_{g\in G}%
\mathcal{D}$ and the trivial action of $G$ on $\mathcal{D}$. However,
considering these actions leads to a new characterization of amenability:

\begin{thmintro}
\label{thm:B} Let $G$ be a countable discrete group, and let $\mathcal{D}$
be a finite strongly self-absorbing C*-algebra. Then $G$ is amenable if and
only if the Bernoulli shift $G\curvearrowright \bigotimes\limits_{g\in G}%
\mathcal{D}$ absorbs tensorially (up to cocycle conjugacy) the trivial
action $\mathrm{id}_{\mathcal{Z}}$ on $\mathcal{Z}$.
\end{thmintro}

This result implies, in particular, a weak form of part~(2) of Conjecture~%
\ref{cnjintro}: every nonamenable group admits two non-cocycle conjugate
strongly outer actions on $\mathcal{D}$, namely, the Bernoulli shift and its
stabilization with $\mathrm{id}_{\mathcal{Z}}$.

We obtain stronger results for groups having sufficiently many finite
subquotients. A particular instance of our main result (\autoref%
{thm:UnctblyActs}) confirms part~(2) of Conjecture~\ref{cnjintro} for groups
containing $\mathbb{F}_2$:

\begin{thmintro}
\label{thm:C} Let $G$ be a discrete group containing a copy of $\mathbb{F}_2$%
, and let $A$ be a $\mathcal{Z}$-absorbing tracial C*-algebra. (For example,
a strongly self-absorbing C*-algebra.) Then there exist uncountably many
non-cocycle conjugate, strongly outer actions of $G$ on $A$, acting via
asymptotically inner automorphisms of $A$.
\end{thmintro}

The fact that the actions we construct are pointwise asymptotically inner
implies that these actions are not distinguishable by any kind of $K$- or $%
KK $-theoretical invariant, or even the Cuntz semigroup. The way in which we
distinguish them is via the weak equivalence class of the associated Koopman
representation.

Theorem~\ref{thm:C} is the C*-version of Ioana's result on non-orbit
equivalent actions from \cite{Ioa_orbit_2011}. Epstein later combined
Ioana's result with Gaboriau-Lyon's measurable solution to the von Neumann
problem \cite{GabLyo_measurable_2009} to generalize Ioana's work to \emph{all%
} nonamenable groups. A very interesting and promising problem is to find an
analog of the main result in \cite{GabLyo_measurable_2009} in the context of
strongly outer actions on C*-algebras, or at least for outer actions on $%
\mathcal{R}$. A satisfactory solution would allow us to prove Theorem~\ref%
{thm:C} (and Theorem~\ref{thm:D} below) for all nonamenable groups.

Finally, our methods allow us to replace $\mathcal{Z}$ with the hyperfinite
II$_1$-factor $\mathcal{R}$, thus obtaining the following generalization of
a recent result of Brothier-Vaes:

\begin{thmintro}
\label{thm:D} Let $G$ be a discrete group containing a copy of $\mathbb{F}_2$%
, and let $M$ be a McDuff tracial von Neumann algebra. Then there exist
uncountably many non-cocycle conjugate, strongly outer actions of $G$ on $M$%
, acting via asymptotically inner automorphisms of $M$.
\end{thmintro}

Even in the case when $M=\mathcal{R}$, our proof is more elementary than the
proof of Theorem~B in~\cite{BroVae_families_2015}, and avoids the use of
Popa's advanced techniques from \cite{Pop_strong_2003} and \cite%
{Pop_some_2006}. It should be mentioned that Theorem~B in~\cite%
{BroVae_families_2015} applies not just to groups containing $\mathbb{F}_2$,
but to arbitrary nonamenable groups.

The rest of the paper is organized as follows. In Section~1, we establish a
number of basic facts about subgroups with finite index that will be
important in the later sections. In Section~2, we study the generalized
Bernoulli shift associated with an action $G\curvearrowright^\sigma X$ on a
discrete set $X$, and relate its Koopman representation to the canonical
unitary representation of $G$ on $\ell^2(X)$. In Section~3, we specialize to
a particular family of generalized Bernoulli shifts, obtained from finite
subquotients of $G$. Finally, Section~4 contains the proofs of our main
results (\autoref{thm:TwoActs} and \autoref{thm:UnctblyActs}), from which
Theorems~\ref{thm:B}, \ref{thm:C} and \ref{thm:D} follow.

\textbf{Acknowledgements:} The authors thank Todor Tsankov for a
conversation on quasiregular representations, as well as Hannes Thiel for
very valuable feedback.

\section{Quasiregular representations}

Given a C*-algebra $D$, we will denote by $\mathcal{S}(D)$ the space of
states over $D$. We begin by recalling the notion of quasiregular
representation.

\begin{df}
\label{df:qrrep} Let $G$ be a discrete group, and let $H$ be subgroup. We
denote by $\lambda_{G/H}\colon G\to{\mathcal{U}}(\ell^2(G/H))$ the unitary
representation induced by the canonical left translation action of $G$ on $%
G/H$. We call $\lambda_{G/H}$ the \emph{quasiregular representation}
associated with $H$.
\end{df}

Let $G$ be a discrete group and let $\mu\colon G\to {\mathcal{U}}({\mathcal{H%
}}_\mu)$ and $\nu\colon G\to {\mathcal{U}}({\mathcal{H}}_\nu)$ be unitary
representations. Recall that $\mu$ is said to be \emph{(unitarily) contained}
in $\nu$, written $\mu\subseteq \nu$, if there exists an isometry $%
\varphi\colon {\mathcal{H}}_\mu\to {\mathcal{H}}_\nu$ satisfying $%
\varphi\circ\mu_g=\nu_g\circ \varphi$ for all $g\in G$. When $\varphi $ can
be chosen to be a unitary, we say that $\mu$ and $\nu$ are \emph{unitarily
equivalent}, and write $\mu\cong \nu$.


\begin{lma}
\label{lma:IndexMultipl} Let $G$ be a discrete group, and let $H_1,\ldots,
H_n$ be subgroups of $G$ whose indices in $G$ are finite. Set $H=H_1\cap
\cdots \cap H_n$. Then:

\begin{enumerate}
\item $H$ has finite index in $G$, and $[G:H]\leq [G:H_1]\cdots [G:H_n]$. If
the indices of $H_1,\ldots, H_n$ in $G$ are relatively prime, then equality
holds.

\item We have $\lambda_{G/H}\subseteq \lambda_{G/H_1}\oplus \cdots\oplus
\lambda_{G/H_n}$, and this containment is an equivalence whenever the
indices of $H_1,\ldots, H_n$ in $G$ are relatively prime.
\end{enumerate}
\end{lma}

\begin{proof}
It is enough to prove both parts for $n=2$. We begin with some notation.
Consider the diagonal action on $\ell^2(G/H_1)\otimes \ell^2(G/H_2)\cong
\ell^2(G/H_1\times G/H_2)$ via $\lambda_{G/H_1}\otimes \lambda_{G/H_2}$. Set 
$x=(H_1,H_2)\in G/H_1\times G/H_2$. Then the stabilizer of $x$ is precisely $%
H_1\cap H_2$. Define a map $\psi\colon G\to G/H_1\times G/H_2$ by $%
\psi(g)=g\cdot x$ for all $g\in G$.

(1). Since $\psi$ is the orbit map associated with $x$, we have 
\begin{equation*}
[G:H_1\cap H_2]=|\psi(x)|\leq [G:H_1][G:H_2].
\end{equation*}
When $[G:H_1]$ and $[G:H_2]$ are coprime, one checks that $\psi$ is
surjective, so we get equality.

(2). Consider the induced map $\widehat{\psi}\colon G/(H_1\cap H_2)\to
G/H_1\times G/H_2$, which is given by $\widehat{\psi}(g(H_1\cap
H_2))=(gH_1,gH_2)$ for all $g\in G$. Define $\varphi\colon \ell^2(G/(H_1\cap
H_2))\to \ell^2(G/H_1\times G/H_2)$ on the canonical orthonormal basis by
setting $\varphi(\delta_{g(H_1\cap H_2)})=\delta_{\widehat{\psi}(g)}$ for
all $g\in G$. It is clear that $\varphi$ is an isometry, and that it is a
unitary if $[G:H_1]$ and $[G:H_2]$ are coprime. It remains to show that $%
\varphi$ intertwines $\lambda_{G/(H_1\cap H_2)}$ and $\lambda_{G/H_1}\oplus
\lambda_{G/H_2}$. Given $g,k\in G$, we have 
\begin{align*}
\left(\lambda_{G/H_1}(k)\oplus
\lambda_{G/H_2}(k)\right)\left(\varphi(\delta_{g(H_1\cap H_2)})\right)&=
\left(\lambda_{G/H_1}(k)\oplus
\lambda_{G/H_2}(k)\right)\left(\delta_{(gH_1,gH_2)}\right) \\
&=\delta_{kgH_1,kgH_2} \\
&= \varphi(\delta_{kg(H_1\cap H_2)}) \\
&=\varphi\left( \lambda_{G/(H_1\cap H_2)}(k)(\delta_{g(H_1\cap H_2)})\right).
\end{align*}
This finishes the proof.
\end{proof}

Our final lemma is well-known, so we only sketch the proof; see also \cite%
{KecTsa_amenable_2008}.

\begin{lma}
\label{lma:QregRepContained} Let $G$ be a discrete group, let $S$ be a
subgroup of $G$, and let $H$ be a subgroup of $S$ with $[S : H]<\infty$.
Then $\lambda_{G/S}\subseteq \lambda_{G/H}$.
\end{lma}

\begin{proof}
Let $\pi\colon G/H\to G/S$ be the canonical quotient map. Then $%
\varphi\colon \ell^2(G/S)\to \ell^2(G/H)$ given by $\varphi(\xi)=\frac{1}{%
\sqrt{[S:H]}}\xi\circ\pi$ for all $\xi\in\ell^2(G/S)$, is an equivariant
isometry.
\end{proof}

In particular, if $S$ is a \emph{finite} subgroup of $G$, then $%
\lambda_{G/S}\subseteq \lambda_G$.

\section{Generalized Bernoulli shifts}

In this section, we study a class of group actions on C*-algebras which are
obtained from permutation actions of $G$ on (discrete) sets. First, we need
to discuss how the GNS construction behaves with respect to infinite tensor
products.

\subsection{Infinite tensor products and the GNS construction.}

We briefly review the GNS construction.

\begin{df}
\label{df:GNS} Let $D$ be a C*-algebra, and let $\phi\colon D\to {\mathbb{C}}
$ be a state on $D$. Define $\langle \cdot,\cdot\rangle_\phi \colon D\times
D\to{\mathbb{C}}$ by $\langle a,b\rangle_\phi=\phi(a^*b)$ for all $a,b\in D$%
. Let ${\mathcal{H}}^D_\phi$ denote the Hilbert space obtained as the
Hausdorff completion of $D$ in the seminorm induced by $\langle\cdot,\cdot%
\rangle_\phi$. We denote by $\iota^D_\phi\colon D\to {\mathcal{H}}^D_\phi$
the canonical map with dense image. When $D$ is clear from the context, we
will simply write ${\mathcal{H}}_\phi$ and $\iota_\phi$.

There is a canonical representation of $D$ on ${\mathcal{H}}_\phi$ by left
multiplication, and we denote by $\overline{D}^\phi$ the weak closure of the
image of $D$ in this representation. Then $\phi$ extends to a faithful
normal state on $\overline{D}^\phi$, which we usually denote again by $\phi$.
\end{df}

When $D$ is a von Neumann algebra and $\phi$ is a normal, faithful state on
it, then $\overline{D}^{\phi}=D$ and the extension of $\phi$ is just $\phi$
again.


We turn to infinite tensor products of Hilbert spaces.

\begin{df}
\label{df:InfTensProdHs} Let ${\mathcal{H}}$ be a Hilbert space, let $\eta\in%
{\mathcal{H}}$ be a unit vector, and let $X$ be a discrete set. We define
the tensor product of ${\mathcal{H}}$ over $X$ (along $\eta$) to be the
completion of 
\begin{equation*}
\mathrm{span}\left\{\bigotimes\limits_{x\in X}\xi_x\colon \xi_x\in {\mathcal{%
H}}, \mbox{ and } \xi_x=\eta \mbox{ for all but finitely many }x\in
X\right\},
\end{equation*}
in the norm induced by the pre-inner product given by 
\begin{equation*}
\left\langle \bigotimes\limits_{x\in X}\xi_x,\bigotimes\limits_{x\in
X}\zeta_x\right\rangle=\prod_{x\in X}\langle\xi_x,\zeta_x\rangle.
\end{equation*}
(Observe that all but finitely many of the multiplicative factors above are
equal to 1, so that the product is indeed finite.)
\end{df}

It will be convenient to have a description of an orthonormal basis of an
infinite tensor product of Hilbert spaces.

\begin{lma}
\label{lma:BasisHs} Let ${\mathcal{H}}$ be a Hilbert space, let $\eta\in{%
\mathcal{H}}$ be a unit vector, and let $X$ be a discrete set. Denote by $%
\kappa $ the dimension of ${\mathcal{H}}$. Let $\{\eta_n\colon n\in \kappa
\} $ be an orthonormal basis for ${\mathcal{H}}$ with $\eta_0=\eta$. Then
the set 
\begin{equation*}
\mathcal{F}=\{\xi\colon X\to \kappa \colon \xi(x)=0 
\mbox{ for all but
finitely many } x\in X\}
\end{equation*}
can be canonically identified with an orthonormal basis for $%
\bigotimes\limits_{x\in X}{\mathcal{H}}$.
\end{lma}

We will need infinite (minimal) tensor products of \emph{unital} C*-algebra
s and infinite (spatial) tensor products of von Neumann algebras (along
states).

Let $D$ be a unital $C^*$-algebra, and let $X$ be a countable set. Write $%
\mathbb{P}_f(X)$ for the set of all finite subsets of $X$, ordered by
inclusion. We define the tensor product $\bigotimes\limits_{x\in X}D$ to be
the direct limit of the minimal tensor products $\bigotimes\limits_{x\in S}D$%
, for $S\in \mathbb{P}_f(X)$, with the canonical connecting maps $%
\iota_{S,T}\colon \bigotimes\limits_{x\in S}D \to \bigotimes\limits_{x\in
T}D $ given by $\iota_{S,T}(d)=d\otimes 1_{T\setminus S}$ for $d\in
\bigotimes\limits_{x\in S}D$, whenever $S,T\in \mathbb{P}_f(X)$ satisfy $%
S\subseteq T$. If $\phi$ is a state on $D$, then the direct limit of the
states $\bigotimes\limits_{x\in S}\phi$, for $S\in\mathbb{P}_f(X)$, defines
a state on $\bigotimes\limits_{x\in X}D$, which we denote by $%
\bigotimes\limits_{x\in X}\phi$. If $\phi$ is a trace, then so is $%
\bigotimes\limits_{x\in X}\phi$.

\begin{rem}
The tensor product $\bigotimes\limits_{x\in X}D$ is canonically isomorphic
to the C*-algebra of operators on $\bigotimes\limits_{x\in X}{\mathcal{H}}$
generated by the operators of the form $\bigotimes_{x\in X}a_x$, where $%
a_x\in D$ and $a_x=1_D$ for all but finitely many $x\in X$.
\end{rem}

If $M$ is a von Neumann algebra and $\phi$ is a normal state on it, the
(spatial) tensor product of $M$ over $X$ with respect to $\phi$ is defined
similarly: by considering $M$ as a unital C*-algebra, we construct its
infinite C*-algebraic tensor product $\bigotimes\limits_{x\in X}M$ as
before. Then $\bigotimes\limits_{x\in X}\phi$ is a state on it, and we
define the infinite spatial tensor product $\overline{\bigotimes\limits_{x%
\in X}} M$ to be the weak-$\ast$ closure of $\bigotimes\limits_{x\in X} M$
in the GNS representation of $\bigotimes\limits_{x\in X}\phi$.

For notational convenience, we will not make a distinction between the
minimal tensor product of C*-algebras and the spatial tensor product of von
Neumann algebras. Here, we will denote both of them by $\otimes$.

Next, we show that GNS constructions commute with infinite tensor products.

\begin{thm}
\label{thm:GNSinfTensProd} Let $D$ be either a unital $C^*$-algebra\ or a
von Neumann algebra, let $X$ be a discrete set, and let $\phi\colon D\to {%
\mathbb{C}}$ be a (normal) state. Set $\eta=\iota_\phi(1_D)\in {\mathcal{H}}%
_\phi$. We will abbreviate $\phi^X=\bigotimes\limits_{x\in X}\phi$. Then
there is a canonical unitary 
\begin{equation*}
u\colon \bigotimes_{x\in X}{\mathcal{H}}_\phi\to {\mathcal{H}}_{\phi^X}
\end{equation*}
determined on a dense subset by 
\begin{equation*}
u\left(\bigotimes_{x\in X}\iota_\phi(a_x)\right)=
\iota_{\phi^X}\left(\bigotimes_{x\in X}a_x\right),
\end{equation*}
where $a_x\in D$ for all $x\in X$, and $a_x=1_D$ for all but finitely many $%
x\in X$. (The tensor product $\bigotimes\limits_{x\in X}{\mathcal{H}}_\phi$
is taken along $\eta$.)
\end{thm}

\begin{proof}
We only outline the proof it in the case that $D$ is a C*-algebra, since the
case of a von Neumann algebra is essentially identical. Let $x\in X$ and
write $\psi_x\colon D\to \bigotimes\limits_{x\in X}D$ for the $x$-th tensor
factor embedding. Since $\phi=\phi^X\circ\psi_x$, it follows that $\psi_x$
induces a Hilbert space isometry $u_x\colon {\mathcal{H}}_\phi\to {\mathcal{H%
}}_{\phi^X}$. Then $u_x$ satisfies $u_x\circ \iota_\phi=\iota_{\phi^X}\circ
\psi_x$.

Denote by $\theta_x\colon {\mathcal{H}}_\phi\to \bigotimes\limits_{x\in X}{%
\mathcal{H}}_\phi$ the canonical isometry as the $x$-th tensor factor. By
the universal property of the tensor product, there exists a bounded linear
map $u\colon \bigotimes\limits_{x\in X}{\mathcal{H}}_\phi\to{\mathcal{H}}%
_{\phi^X}$ satisfying $u_x=u\circ\theta_x$ for all $x\in X$. It is then easy
to check that $u$ is a unitary, and that it satisfies the identity in the
statement. We omit the details.
\end{proof}

\subsection{Generalized Bernoulli shifts.}

We begin by introducing some useful notation and terminology. Let $G$ be a
discrete group, and let $X$ be a discrete set. By a \emph{(set) action} of $%
G $ on $X$ we mean a homomorphism $\sigma\colon G\to \mathrm{Perm}(X)$ from $%
G$ to the group $\mathrm{Perm}(X)$ of permutations of $X$. We usually
abbreviate this to $G\curvearrowright^{\sigma} X$. Also, for $g\in G$ and $%
x\in X$, we write $g\cdot x$ for $\sigma(g)(x)$.

\begin{df}
\label{df:UsGenBernoulli} Let $G$ be a countable group, let $X$ be a
countable set, and let $G\curvearrowright^{\sigma} X$ be an action. Endow $X$
with the counting measure, and let $D$ be a unital $C^*$-algebra\ or a
tracial von Neumann algebra.

\begin{enumerate}
\item The \emph{unitary representation associated with $\sigma$} is the
unitary representation $u_\sigma \colon G\to {\mathcal{U}}(\ell^2(X))$ given
by $u_\sigma(g)(\delta_{x})=\delta_{g^{-1}\cdot x}$ for all $g\in G$ and all 
$x\in X$.

\item The \emph{generalized Bernoulli shift associated with $\sigma$} is the
action $\beta_{\sigma,D}\colon G\to {\mathrm{Aut}}\left(\otimes_{x\in
X}D\right)$ given by permuting the tensor factors according to $%
G\curvearrowright^{\sigma} X$.
\end{enumerate}
\end{df}

\begin{nota}
\label{nota:leftTransRegRep} Let $G$ be a discrete group. We will denote by $%
\sigma_G$ the action of left translation $G\curvearrowright G$, so that $%
u_{\sigma_G}$ is the left regular representation $\lambda_G\colon G\to{%
\mathcal{U}}(\ell^2(G))$. Similarly, if $H$ is a subgroup of $G$, we will
denote by $\sigma_{G/H}$ the canonical action $G\curvearrowright G/H$ by
left translation of left cosets, so that $u_{\sigma_{G/H}}$ is the
quasiregular representation $\lambda_{G/H}\colon G\to {\mathcal{U}}%
(\ell^2(G/H))$ from \autoref{df:qrrep}.
\end{nota}

We will also need the Koopman construction, which is a way of obtaining
unitary representations from group actions.

\begin{df}
\label{df:Koopman} Let $G$ be a countable group, let $(D,\phi)$ be a unital $%
C^*$-algebra\ with a state $\phi$, and let $\alpha\colon G\to{\mathrm{Aut}}%
(D)$ be a group action satisfying $\phi\circ\alpha_g=\phi$ for all $g\in G$.

\begin{itemize}
\item The \emph{Koopman representation of $\alpha$ (with respect to $\phi$)}
is the unitary representation $\kappa_\phi (\alpha)\colon G\to{\mathcal{U}}({%
\mathcal{H}}_\phi)$ determined by $\kappa(\alpha)_g(\iota_\phi(a))=\iota_%
\phi(\alpha_g(a))$ for all $g\in G$ and all $a\in D$.

\item The \emph{reduced Koopman representation of $\alpha$ (with respect to $%
\phi$)}, denoted by $\kappa_\phi(\alpha)^{(0)}$, is the restriction of $%
\kappa(\alpha)$ to the orthogonal complement of $\iota_\phi(1_D)$.
\end{itemize}

Since the state $\phi$ will usually be fixed, we will often omit it in the
notation of the Koopman representation. (The only exception to this will
occur in \autoref{prop:sRpCocConj}.)
\end{df}

We denote by $1_G$ the trivial unitary representation of $G$ on ${\mathbb{C}}
$.

\begin{rem}
In the notation of the definition above, $\kappa(\alpha)$ is unitarily
equivalent to $\kappa(\alpha)^{(0)}\oplus 1_G$.
\end{rem}

We proceed to collect some elementary lemmas that will be needed later. If $%
\alpha\colon G\to{\mathrm{Aut}}(D,\phi)$ is a state-preserving action of a
discrete group $G$ on a C*-algebra $D$ with a distinguished state $\phi $,
we write $\overline{\alpha}^\phi\colon G\to{\mathrm{Aut}}(\overline{D}%
^\phi,\phi)$ for its weak extension. It is clear that the (reduced) Koopman
representations of $\alpha$ and of $\overline{\alpha}^\phi$ (with respect to 
$\phi$) agree in a natural way.

The following notation will be useful throughout.

\begin{nota}
\label{nota:kappaSgma} Let $G$ be a countable group, let $%
G\curvearrowright^{\sigma} X$ be an action of $G$ on a countable set $X$,
and let ${\mathcal{H}}$ be a Hilbert space with a distinguished unit vector $%
\eta$. We denote by $\kappa^{{\mathcal{H}}}_\sigma\colon G\to {\mathcal{U}}%
(\otimes_{x\in X}{\mathcal{H}})$ the unitary representation given by
permuting the tensor factors according to $G\curvearrowright^{\sigma} X$.
When ${\mathcal{H}}$ is clear from the context (as it usually will), we
simply write $\kappa_\sigma$. We define $\kappa_\sigma^{(0)}$ to be the
restriction of $\kappa_\sigma$ to the orthogonal complement of $\eta$ in $%
\bigotimes\limits_{x\in X}{\mathcal{H}}$.
\end{nota}

The notation for $\kappa_\sigma$ and $\kappa_\sigma^{(0)}$ is justified by
parts~(2) and~(3) below.

\begin{lma}
Let $G$ be a countable group, let $(D,\phi)$ be a unital $C^*$-algebra\
endowed with a state $\phi$, and let $G\curvearrowright^{\sigma} X$ be an
action of $G$ on a countable set $X$.

\begin{enumerate}
\item The actions $\overline{\beta}^{\phi}_{\sigma,D}$ and $\beta_{\sigma,%
\overline{D}^\phi}$ are both canonically conjugate.

\item The Koopman representations of $\beta_{\sigma,D}$, of $\overline{\beta}%
^{\phi}_{\sigma,D}$, and of $\beta_{\sigma,\overline{D}^{\phi}}$ are all
canonically unitarily equivalent to $\kappa_\sigma$.

\item The reduced Koopman representations of $\beta_{\sigma,D}$, of $%
\overline{\beta}^{\phi}_{\sigma,D}$ and of $\beta_{\sigma,\overline{D}%
^{\phi}}$ are all canonically unitarily equivalent to $\kappa_\sigma^{(0)}$.
\end{enumerate}
\end{lma}

\begin{proof}
Observe that $\phi$ extends to a (normal, faithful) state on $\overline{D}%
^{\phi}$, which we denote again by $\phi$. It is easy to check that the GNS
constructions of $D$ and $\overline{D}^{\phi}$ (with respect to $\phi$)
agree in a natural way. Using this observation together with \autoref%
{thm:GNSinfTensProd}, the proof of the lemma is straightforward, and we omit
the details.
\end{proof}

\begin{lma}
\label{lma:UnionSn} For $n\in{\mathbb{N}}$, let $G\curvearrowright^{%
\sigma_n}X_n$ be actions, and let $G\curvearrowright^{\sigma} X$ denote the
disjoint union of them (with $X=\bigsqcup\limits_{n\in{\mathbb{N}}}X_n$).
Let $(D,\phi)$ be a unital $C^*$-algebra\ endowed with a state $\phi$.

\begin{enumerate}
\item The representations $u_\sigma$ and $\bigoplus\limits_{n\in{\mathbb{N}}%
}u_{\sigma_n}$ are canonically unitarily equivalent.

\item The actions $\beta_{\sigma,D}$ and $\bigotimes\limits_{n\in{\mathbb{N}}%
}\beta_{\sigma_n,D}$ are canonically conjugate.

\item The representations $\kappa_\sigma$ and $\bigotimes\limits_{n\in{%
\mathbb{N}}}\kappa_{\sigma_n}$ are canonically unitarily equivalent.

\item The representations $\kappa^{(0)}_\sigma$ and $\bigotimes\limits_{n\in{%
\mathbb{N}}}\kappa^{(0)}_{\sigma_n}$ are canonically unitarily equivalent.
\end{enumerate}
\end{lma}

\begin{proof}
All parts follow from the facts that $\ell^2(\sqcup_{n\in{\mathbb{N}}}X_n)$
is canonically isometrically isomorphic to $\oplus_{n\in{\mathbb{N}}%
}\ell^2(X_n)$, as witnessed by a unitary intertwining the representations $%
\sigma$ and $\oplus_{n\in{\mathbb{N}}}\sigma_n$; and that $\otimes_{x\in X}D$
is canonically isomorphic to $\otimes_{n\in{\mathbb{N}}}\otimes_{x\in X_n}D$%
, as witnessed by an isomorphism that intertwines $\beta_{\sigma,D}$ and $%
\otimes_{n\in{\mathbb{N}}}\beta_{\sigma_n,D}$.
\end{proof}


We need to recall the following definition.

\begin{df}
\label{df:2norm} Let $D$ be a C*-algebra, and let $\phi\in \mathcal{S}(D)$
be a state on it. Given $a\in D$, we set $\|a\|_{2,\phi}=\phi(a^*a)^{1/2}$.
\end{df}

\begin{rem}
If $D$ is a C*-algebra and $\phi\in \mathcal{S}(D)$, then $%
\|ab\|_{2,\phi}\leq \|a\|\|b\|_{2,\phi}$ for all $a,b\in D$.
\end{rem}

By definition, ${\mathcal{H}}_\phi$ is the (Hausdorff) completion of $D$
with respect to $\|\cdot\|_{2,\phi}$. On the other hand, with $D_1$ denoting
the norm-unit ball of $D$, Kaplansky's density theorem implies that the weak
closure $\overline{D}^\phi$ of $D$ in the GNS construction for $\phi$ can be
identified with $\overline{\mbox{span}}\left(\overline{D_1}%
^{\|\cdot\|_{2,\phi}}\right)\subseteq {\mathcal{B}}({\mathcal{H}}_\phi)$.

The following Rokhlin-type property will allow us to relate the unitary
representations $u_\sigma$ from \autoref{df:UsGenBernoulli} and $%
\kappa_\sigma$ from \autoref{nota:kappaSgma}; see \autoref{prop:sRpWkCtmnt}.

\begin{df}
\label{df:sRp} Let $G$ be a countable group, let $D$ be a C*-algebra, let $%
\phi\in \mathcal{S}(D)$, and let $G\curvearrowright^{\sigma} X$ be an action
of $G$ on a countable set $X$. An action $\alpha\colon G\to{\mathrm{Aut}}(D)$
preserving $\phi$ is said to have the \emph{$\sigma$-Rokhlin property (with
respect to $\phi$)} if there exist projections $p_x^{(n)}\in \overline{D}%
^\phi$, for $x\in X$ and $n\in{\mathbb{N}}$, satisfying the following
conditions:

\begin{enumerate}
\item $\lim\limits_{n\to\infty}\|\overline{\alpha}^\phi_g(p_x^{(n)})-p_{g%
\cdot x}^{(n)}\|_{2,\phi}\to 0$ for every $g\in G$ and $x\in X$.

\item $\lim\limits_{n\to\infty}\|ap_x^{(n)}-p_{x}^{(n)}a\|_{2,\phi}\to 0$
for every $a\in D$ and $x\in X$.

\item $\phi(p_x^{(n)}p_y^{(n)})=\phi(p_x^{(n)})\phi(p_y^{(n)})$ for all $%
x,y\in X$ with $x\neq y$ and for all $n\in{\mathbb{N}}$.

\item $\phi(p_x^{(n)})$ does not depend on $x$ or on $n$, and belongs to ${%
\mathbb{Q}}\cap (0,1)$.
\end{enumerate}
\end{df}

Condition (3) asserts that the different projections are ``independent",
while Condition (4) is a non-triviality condition, since any action
satisfies conditions (1), (2) and (3) above with $p_x^{(n)}=0$ for all $x\in
X$ and all $n\in{\mathbb{N}}$.

Having a projection whose value on $\phi$ is rational and nontrivial is a
mild condition, although it is not automatic, even if $\phi$ is a trace. For
example, ${\mathbb{C}}\oplus{\mathbb{C}}$ with trace $\tau$ determined by $%
\tau(1,0)=\sqrt{2}/2$, has no projections of rational trace other than 0 and
1.

In contrast to the example above, our next result, which is of independent
interest, shows that every infinite-dimensional tracial von Neumann algebra
always has projections of arbitrarily small (rational) trace.

\begin{thm}
\label{thm:InfDimvNaPjns} Let $M$ be an infinite dimensional von Neumann
algebra, let $\phi\in \mathcal{S}(M)$ be a normal faithful state, and let $%
\mathcal{P}(M)$ be the set of its projections. Then there exists $%
0<\varepsilon\leq 1$ such that $[0,\varepsilon]\subseteq \phi(\mathcal{P}%
(M)) $.
\end{thm}

\begin{proof}
Without loss of generality, we will assume that $M$ has separable predual.
We divide the proof into three cases.

\emph{Case 1: the center of $M$ is infinite dimensional.} Then there are a
standard probability space $(Y,\mu )$ whose measure is not concentrated on
finitely many points, and an isomorphism $(Z(M),\phi )\cong L^{\infty
}(X,\mu )$. It suffices to show that for every sufficiently small $r\in
\lbrack 0,1]$, there exists a measurable subset $E\subseteq X$ with $\mu
(E)=r$. (This implies $\phi ({\mathcal{P}}(M))=[0,1]$.) Again we consider
two cases.

\begin{itemize}
\item \emph{Case 1.1: $Y$ is purely atomic.} Let $r\in [0,1]$. Find $y_1\in
A $ such that $\mu(\{y_1\})\leq r$, and set $Y_1=\{y_1\}$. If $\mu(Y_1)< r$,
find $y_2\in Y$ such that $\mu(\{y_1, y_2\})\leq r$, and set $Y_2=Y_1\cup
\{y_2\}$. This process terminates if there are $y_1,\ldots,y_n\in Y$ such
that $\mu(\{y_1,\ldots,y_n\}=r$. Otherwise there is a sequence $(Y_n)_{n\in{%
\mathbb{N}}}$ of subsets of $Y$ such that $(\chi_{Y_n})_{n\in{\mathbb{N}}}$
is a Cauchy sequence in $L^1(Y,\mu)$. Then its limit must be the
characteristic function of a set $E\subseteq Y$ of measure $r$.

\item \emph{Case 1.2: $Y$ is not purely atomic.} In this case, after
replacing $Y$ with a Borel subset, and then renormalizing the restriction of
the measure, we can assume that $(Y,\mu)$ is atomless, and hence isomorphic
to $[0,1]$ with the Lebesgue measure. In this case, the conclusion is
obvious.

\end{itemize}

\emph{Case 2: $M$ is a factor.} Since $M$ is infinite dimensional, it
contains the hyperfinite II$_1$-factor $\mathcal{R}$ as a (unital)
subfactor. Upon restricting $\phi$ to any infinite dimensional abelian
subalgebra of $\mathcal{R}$, it follows that $\phi({\mathcal{P}}(M))=[0,1]$.

\emph{Case 3: the center of $M$ is finite dimensional.} In this case, $M$ is
a direct sum of factors $M\cong N_1\oplus\cdots\oplus N_m$. Since $M$ is
infinite dimensional, there exists $k=1,\ldots,m$ such that $N_k$ is
infinite dimensional. In this case, with $1_k$ denoting the unit of $N_k$ in 
$M$, it follows from Case 2 that $[0,\phi(1_k)]\subseteq \phi(\mathcal{P}%
(M)) $.
\end{proof}

Next, we construct examples of actions with the $\sigma$-Rokhlin property
using generalized Bernoulli shifts.

\begin{prop}
\label{prop:BetaHasSRp} Let $G$ be a countable group, let $D$ be a unital $%
C^*$-algebra, and let $\phi\in\mathcal{S}(D)$ be a state for which $%
\overline{D}^\phi$ is not one-dimensional. Let $G\curvearrowright^{\sigma} X$
be an action of $G$ on a countable set $X$. Then $\bigotimes\limits_{n\in{%
\mathbb{N}}}\beta_{\sigma,D}$ has the $\sigma$-Rokhlin property. Moreover,
if $\alpha\colon G\to{\mathrm{Aut}}(A)$ is any action on a C*-algebra $A$,
and $\phi_A$ is an $\alpha$-invariant state on it, then $\alpha\otimes
\bigotimes\limits_{n\in{\mathbb{N}}}\beta_{\sigma,D}$ has the $\sigma$%
-Rokhlin property (with respect to $\phi_A\otimes \bigotimes\limits_{n\in{%
\mathbb{N}}}\phi$).
\end{prop}

\begin{proof}
Observe that $\otimes_{n\in{\mathbb{N}}}\beta_{\sigma,D}$ is conjugate to $%
\otimes_{n\in{\mathbb{N}}}\beta_{\sigma,\otimes_{n\in{\mathbb{N}}}D}$, and
that the weak closure of $\otimes_{n\in{\mathbb{N}}}D$ with respect to $\phi$
is infinite dimensional. Thus we may assume, without loss of generality,
that $\overline{D}^\phi$ is infinite dimensional.

Use \autoref{thm:InfDimvNaPjns} to fix a projection $p\in \overline{D}^\phi$
with $\phi(p)\in \mathbb{Q}\cap (0,1)$. For $n\in{\mathbb{N}}$ and $x\in X$,
denote by 
\begin{equation*}
\iota_{x}^{(n)}\colon \overline{D}^\phi \to \bigotimes_{n\in{\mathbb{N}}%
}\bigotimes_{x\in X}\overline{D}^\phi
\end{equation*}
the canonical embedding into the $(n,x)$-th tensor factor, and set $%
p_x^{(n)}=\iota_x^{(n)}(p)$. Since $(\overline{\beta}^\phi_{\sigma,D})_g%
\circ\iota_x^{(n)}=\iota_{g\cdot x}^{(n)}$ for all $g\in G$, all $x\in X$
and all $n\in{\mathbb{N}}$, condition (1) is guaranteed. The remaining
conditions are easily checked.

For the last assertion, if $\{p_x^{(n)}\colon x\in X, n\in{\mathbb{N}}\}$
are projections as in \autoref{df:sRp} for $\bigotimes\limits_{n\in{\mathbb{N%
}}}\overline{\beta}^\phi_{\sigma,D}$, then $\{1\otimes p_x^{(n)}\colon x\in
X, n\in{\mathbb{N}}\}$ witness the fact that $\alpha\otimes
\bigotimes\limits_{n\in{\mathbb{N}}}\beta_{\sigma,D}$ has the $\sigma$%
-Rokhlin property with respect to $\phi_A\otimes \bigotimes\limits_{n\in{%
\mathbb{N}}}\phi$.
\end{proof}

In particular, if $\sigma$ is conjugate to its infinite amplification $%
\sigma\times{\mathrm{id}}_{{\mathbb{N}}}$, then $\beta_{\sigma,D}$ always
has the $\sigma$-Rokhlin property.

We recall the definition of weak containment for representations in the
sense of Zimmer; see also Definition~1.3 in~\cite{BurKec_weak_2016}.

\begin{df}
\label{df:wkcont} Let $G$ be a discrete group, and let $\mu\colon G\to{%
\mathcal{U}}({\mathcal{H}}_\mu)$ and $\nu\colon G\to{\mathcal{U}}({\mathcal{H%
}}_\nu)$ be unitary representations. We say that \emph{$\mu$ is weakly
contained in $\nu$ in the sense of Zimmer}, in symbols $\mu\prec_{\mathrm{Z}%
}\nu$, if for any $\varepsilon>0$, for any $\xi_1,\ldots,\xi_n\in {\mathcal{H%
}}_\mu$, for any finite subset $F\subseteq G$, and for any $\varepsilon>0$,
there exist $\eta_1,\ldots,\eta_n\in{\mathcal{H}}_\nu$ satisfying 
\begin{equation*}
\left|\langle \mu_g(\xi_j),\xi_k\rangle - \langle
\nu_g(\eta_j),\eta_k\rangle\right|<\varepsilon
\end{equation*}
for all $g\in F$ and for all $j,k=1,\ldots,n$.

Finally, we say that $\mu$ and $\nu$ are \emph{weakly equivalent in the
sense of Zimmer}, written $\mu\sim_{\mathrm{Z}}\nu$, if $\mu\prec_{\mathrm{Z}%
}\nu$ and $\nu\prec_{\mathrm{Z}}\mu$.
\end{df}

We will not be using the standard notion of weak containment, which is
weaker. (In fact, one can show $\mu $ is weakly contained in $\nu $ in the
usual sense if and only if $\mu $ is weakly contained, in the sense of
Zimmer, in the infinite amplification of $\nu $.) It is obvious that, when $%
G $ is countable, $\mu \prec _{\mathrm{Z}}\nu $ if and only if, for every
separable subrepresentation $\mu ^{\prime }$ of $\mu $, we have $\mu ^{\prime }\prec
_{\mathrm{Z}}\nu $.

Below, we present a characterization of weak containment in the sense of
Zimmer that will be convenient for our purposes. We need a short discussion
on ultrapowers of unitary representations first.

Let $\mathcal{U}$ be a free ultrafilter over an index set $I$ and let ${%
\mathcal{H}}$ be a Hilbert space. Set 
\begin{equation*}
{\mathcal{H}}^\mathcal{U}=\ell^{\infty}(I,{\mathcal{H}})/\{(\xi_j)_{j\in
I}\in \ell^\infty(I,{\mathcal{H}})\colon \lim_{j\to\mathcal{U}}\|\xi_j\|=0\},
\end{equation*}
endowed with the quotient norm. The class in ${\mathcal{H}}^\mathcal{U}$ of
a sequence $\xi\in\ell^\infty(I,{\mathcal{H}})$ is denoted by $[\xi]$. Then $%
{\mathcal{H}}^\mathcal{U}$ is a Hilbert space with respect to 
\begin{equation*}
\langle [\xi],[\eta]\rangle=\lim_{j\to\mathcal{U}}\langle \xi_j,\eta_j\rangle
\end{equation*}
for all $\xi,\eta\in \ell^\infty(I,{\mathcal{H}})$. If $\nu\colon G\to{%
\mathcal{U}}({\mathcal{H}})$ is a unitary representation of a discrete group 
$G$ on ${\mathcal{H}}$, then there is an induced representation $\nu^%
\mathcal{U}\colon G\to{\mathcal{U}}({\mathcal{H}}^\mathcal{U})$ given by $%
\nu^\mathcal{U}_g([\xi])=[(\nu_g(\xi_j))_{j\in I}]$ for all $g\in G$ and all 
$\xi\in\ell^\infty(I,{\mathcal{H}})$.

\begin{rem}
\label{rem:UltrapowHsSum} Adopt the notation from the discussion above. If $%
\nu_1$ and $\nu_2$ are unitary representations of $G$ on Hilbert spaces ${%
\mathcal{H}}_1$ and ${\mathcal{H}}_2$, respectively, it is easy to verify
that $(\nu_1\oplus\nu_2)^\mathcal{U}$ is naturally unitarily conjugate to $%
\nu_1^\mathcal{U}\oplus\nu_2^\mathcal{U}$. The analogous statement for
infinite direct sums is not true.
\end{rem}

The following result is probably known to the experts. Since we were not
able to find it in the literature, we provide a short proof.

\begin{prop}
\label{prop:charactZimmer} Let $G$ be a countable discrete group, and let $%
\mu \colon G\rightarrow {\mathcal{U}}({\mathcal{H}}_{\mu })$ and $\nu \colon
G\rightarrow {\mathcal{U}}({\mathcal{H}}_{\nu })$ be unitary
representations, where $\mathcal{H}_{\mu }$ is separable. Let $\mathcal{U}$
be a free ultrafilter on $\mathbb{N}$. Then $\mu \prec _{\mathrm{Z}}\nu $ if
and only if $\mu \subseteq \nu ^{\mathcal{U}}$.
\end{prop}

\begin{proof}
Both directions are easy, and follow, for example, by combining \L {}os'
theorem (for the ``if'' implication) and countable saturation of ultrapowers
(for the ``only if'' implication). See, for example, Sections~2.3 and~4.3 in 
\cite{farah_model_2017}.
\end{proof}

In our next preparatory result, the case when $\mu$ is the trivial
representation is well known.

\begin{prop}
\label{prop:IrredRepWkCont} Let $G$ be a countable discrete group, and let $%
\mu \colon G\rightarrow {\mathcal{U}}({\mathcal{H}}_{\mu })$ and $\nu
_{j}\colon G\rightarrow {\mathcal{U}}({\mathcal{H}}_{j})$, for $j=1,\ldots
,n $, be unitary representations. Assume that $\mu $ is irreducible and
finite-dimensional, and that $\mu \prec _{\mathrm{Z}}\nu _{1}\oplus \cdots
\oplus \nu _{n}$. Then there exists $k\in \{1,\ldots ,n\}$ such that $\mu
\prec _{\mathrm{Z}}\nu _{k}$.
\end{prop}

\begin{proof}
Without loss of generality we can assume that $\mathcal{H}_{\mu }$ is
separable. Let $\mathcal{U}$ be any free ultrafilter on $\mathbb{N}$. Use %
\autoref{prop:charactZimmer} to choose an equivariant isometry $\varphi
\colon {\mathcal{H}}_{\mu }\rightarrow ({\mathcal{H}}_{1}\oplus \cdots
\oplus {\mathcal{H}}_{n})^{\mathcal{U}}$ witnessing the fact that $\mu \prec
_{\mathrm{Z}}\nu _{1}\oplus \cdots \oplus \nu _{n}$. We identify $({\mathcal{%
H}}_{1}\oplus \cdots \oplus {\mathcal{H}}_{n})^{\mathcal{U}}$ equivariantly
with ${\mathcal{H}}_{1}^{\mathcal{U}}\oplus \cdots \oplus {\mathcal{H}}_{n}^{%
\mathcal{U}}$ in a canonical way via \autoref{rem:UltrapowHsSum}. For $%
j=1,\ldots ,n$, we denote by $\varphi _{j}\colon {\mathcal{H}}_{\mu
}\rightarrow {\mathcal{H}}_{j}^{\mathcal{U}}$ the composition of $\varphi $
with the canonical projection onto ${\mathcal{H}}_{j}^{\mathcal{U}}$.

Since $\varphi $ is nonzero, there exists $k\in \{ 1,\ldots ,n\}$ such that $%
\varphi _k$ is nonzero. Since $\mu $ is irreducible and finite-dimensional,
by Schur's lemma $\varphi_k$ is a scalar multiple of an isometry. This
concludes the proof. 

\end{proof}

The main use of the $\sigma$-Rokhlin property is given by the combination of
the next proposition with \autoref{prop:sRpCocConj} below.

\begin{prop}
\label{prop:sRpWkCtmnt} Let $G$ be a countable discrete group, let $D$ be a
C*-algebra or a von Neumann algebra, let $\phi \in \mathcal{S}(D)$ be a
(normal) state on it, and let $G\curvearrowright ^{\sigma }X$ be an action
of $G$ on a countable set $X$. If $\alpha \colon G\rightarrow {\mathrm{Aut}}%
(D)$ is a $\phi $-preserving action with the $\sigma $-Rokhlin property,
then 
\begin{equation*}
u_{\sigma }\oplus 1_{G}\prec _{\mathrm{Z}}\kappa (\alpha )\ \ \mbox{ and }\
\ u_{\sigma }\prec _{\mathrm{Z}}\kappa (\alpha )^{(0)}.
\end{equation*}
\end{prop}

\begin{proof}
It suffices to show the second part, since $\kappa(\alpha)\cong
\kappa(\alpha)^{(0)}\oplus 1_G$. Suppose that $\alpha$ has the $\sigma$%
-Rokhlin property, and let $\{p_x^{(n)}\colon x\in X, n\in{\mathbb{N}}\}$ be
a family of projections in $\overline{D}^\phi$ as in \autoref{df:sRp}.
Choose positive integers $r,t\in{\mathbb{N}}$ with $r<t$ satisfying $%
\phi(p_x^{(n)})=r/t$ for all $x\in X$ and all $n\in{\mathbb{N}}$. For $x\in
X $ and $n\in{\mathbb{N}}$, set $\xi_x^{(n)}=r-tp_x^{(n)}\in {\mathcal{H}}%
_\phi $. Since $\phi(r-tp_x^{(n)})=0$, it follows that $\xi_x^{(n)}$ belongs
to the orthogonal complement of the unit. Moreover, for $n\in{\mathbb{N}}$
and $x,y\in X$ with $x\neq y$, we have 
\begin{align*}
\langle \xi_x^{(n)}, \xi_y^{(n)}\rangle&=\phi((r-tp_x^{(n)})^*(r-tp_y^{(n)}))
\\
&= r^2-rt\phi(p_x^{(n)})-rt\phi(p_y^{(n)})+t^2\phi(p_x^{(n)}p_y^{(n)})=0.
\end{align*}
In particular, $\xi_x^{(n)}\perp \xi_y^{(n)}$ whenever $x\neq y$. Similarly, 
\begin{align*}
\langle \xi_x^{(n)}, \xi_x^{(n)}\rangle=
r^2-rt\phi(p_x^{(n)})-rt\phi(p_x^{(n)})+t^2\phi(p_x^{(n)})=r(t-r).
\end{align*}

Additionally, it is immediate that $\lim\limits_{n\to\infty}\|(\kappa(%
\alpha))_g(\xi_x^{(n)})-\xi_{g\cdot x}^{(n)}\|_{{\mathcal{H}}_\phi}=0$ for
all $g\in G$ and all $x\in X$.

Fix $n\in{\mathbb{N}}$. Define a map $\varphi_n\colon \ell^2(X)\to {\mathcal{%
H}}_\phi$ by $\varphi_n(\delta_x)=\xi_x^{(n)}/\sqrt{r(t-r)}$ for all $x\in X$%
. Then the maps $(\varphi_n)_{n\in{\mathbb{N}}}$ are asymptotically
equivariant unitaries, witnessing the fact that $u_\sigma$ is a
subrepresentation of the ultrapower of $\kappa(\alpha)^{(0)}$. By \autoref%
{prop:charactZimmer}, the proof is finished.
\end{proof}

The following definition is standard.

\begin{df}
\label{df:CocConj} Let $G$ be a countable discrete group, let $A$ be a
C*-algebra, and let $\alpha ,\beta \colon G\rightarrow {\mathrm{Aut}}(A)$ be
actions. We say that $\alpha $ and $\beta $ are \emph{cocycle conjugate}, if
there exist an automorphism $\theta \in {\mathrm{Aut}}(A)$ and a function $%
u\colon G\rightarrow {\mathcal{U}}(M(A))$ satisfying 
\[ u_{gh}=u_g\alpha_g(u_h) \ \mbox{ and } \ \beta_g=\theta \circ(\Ad(u_g)\circ\alpha_g)\circ\theta^{-1}\]
for all $g,h\in G$.
for all $g,h\in G$. 
\end{df}

Let $\alpha,\beta\colon G\to{\mathrm{Aut}}(D)$ be actions of a discrete
group $G$ on a C*-algebra $D$. If $\alpha$ and $\beta$ are cocycle
conjugate, there is in general no relationship between $\kappa(\alpha)$ and $%
\kappa(\beta)$, even if they both preserve the same trace. This can be seen,
for example, by letting $\alpha\colon{\mathbb{Z}}_2\to {\mathrm{Aut}}(M_2)$
be the trivial action, and $\beta\colon{\mathbb{Z}}_2\to{\mathrm{Aut}}(M_2)$
be the inner action determined by the order two unitary $\left(
                            \begin{array}{cc}
                              0 & 1 \\
                              1 & 0 \\
                            \end{array}
                          \right)$. 
(In this case, with respect to the unique trace, $\kappa(\alpha)$
is conjugate to $\bigoplus_{j=1}^4 1_{{\mathbb{Z}}_2}$, while $\kappa(\beta)$
is conjugate to $\bigoplus_{j=1}^2 \lambda_{{\mathbb{Z}}_2}$.) In
particular, if $\kappa(\alpha)$ contains a given unitary representation, we
cannot conclude that so does $\kappa(\beta)$. A special case when this is
indeed true is given in the following proposition.

We remark that the following proposition is the only instance in this work
where traces are really necessary, specifically when showing that $\tau
\circ \theta $ is $\beta $-invariant.

\begin{prop}
\label{prop:sRpCocConj} Let $G$ be a countable discrete group, let $A$ be a
C*-algebra, let $\tau \in \mathcal{S}(A)$ be a tracial state, and let $%
\alpha ,\beta \colon G\rightarrow {\mathrm{Aut}}(A)$ be $\tau $-preserving
actions which are cocycle conjugate, and let $\theta \in {\mathrm{Aut}}(A)$
and $u\colon G\rightarrow {\mathcal{U}}(M(A))$ be as in \autoref{df:CocConj}%
. Let $G\curvearrowright ^{\sigma }X$ be an action, and suppose that $\alpha 
$ has the $\sigma $-Rokhlin property with respect to $\tau $. Then $\tau
\circ \theta ^{-1}$ is $\beta $-invariant, and $\beta $ has the $\sigma $%
-Rokhlin property with respect to it. Moreover, 
\begin{equation*}
u_{\sigma }\oplus 1_{G}\prec _{\mathrm{Z}}\kappa _{\tau }(\beta )\ \ 
\mbox{
and }\ \ u_{\sigma }\prec _{\mathrm{Z}}\kappa _{\tau }^{(0)}(\beta ).
\end{equation*}
\end{prop}

\begin{proof}
Find projections $\{p_{x}^{(n)}\colon x\in X,n\in {\mathbb{N}}\}$ as in %
\autoref{df:sRp}, and set $q_{x}^{(n)}=\theta (p_{x}^{(n)})$ for $x\in X$
and $n\in {\mathbb{N}}$. Using asymptotic centrality of these projections,
it is easy to verify that the family $\{q_{x}^{(n)}\colon x\in X,n\in {%
\mathbb{N}}\}$ witnesses the fact that $\beta $ has the $\sigma $-Rokhlin
property with respect to $\tau \circ \theta ^{-1}$.

By \autoref{prop:sRpWkCtmnt}, we have 
\begin{equation*}
u_{\sigma }\oplus 1_{G}\prec _{\mathrm{Z}}\kappa _{\tau \circ \theta
^{-1}}(\beta )\ \ \mbox{ and }\ \ u_{\sigma }\prec _{\mathrm{Z}}\kappa
_{\tau \circ \theta ^{-1}}^{(0)}(\beta ).
\end{equation*}%
Since the Koopman representations of $\beta $ with respect to $\tau $ and $%
\tau \circ \theta ^{-1}$ are unitarily equivalent, the result follows.
\end{proof}

\section{Actions induced by finite subquotients}

In this section, we specialize to Bernoulli shifts associated with a
particular class of actions which are constructed from quasi-regular
representations (see \autoref{df:qrrep}). If $H$ is a subgroup of a discrete
group $G$, then the quasiregular representation $\lambda_{G/H}$ induces a
unital homomorphism $\pi_H\colon C^*(G) \to {\mathcal{B}}(\ell^2(G/H))$, by
universality of $C^*(G)$. When $H$ is normal in $G$, then the image of $%
\pi_H $ is $C^*_{\mathrm{r}}(G/H)$.

\begin{df}
\label{df:separated} Let $F$ be a discrete group, let ${\mathcal{P}}%
\subseteq {\mathbb{N}}$ be a subset of relatively prime numbers. A family $%
\{H_p\}_{p\in{\mathcal{P}}}$ of subgroups of $F$ is said to be \emph{%
separated} if it satisfies the following properties:

\begin{enumerate}
\item[(S.1)] $[F:H_p]=p$ for all $p\in {\mathcal{P}}$, and

\item[(S.2)] for $p\in {\mathcal{P}}$ and $Q\subseteq {\mathcal{P}}$, if $%
\bigcap\limits_{q\in Q} \ker(\pi_{H_{q}})\subseteq \ker(\pi_{H_p})$, then $%
p\in Q$.
\end{enumerate}

We say that $F$ is \emph{separated over $\mathcal{P}$} if it contains a
separated family of subgroups indexed over ${\mathcal{P}}$. Finally, we say
that $F$ is \emph{infinitely separated} if it is separated over an infinite
set ${\mathcal{P}}$ of relatively prime numbers.
\end{df}

Observe that we do not require the subgroups $H_p$ in the definition above
to be normal. Conditions (S.1) and (S.2) seem rather restrictive, but we can
always take ${\mathcal{P}}=\{1\}$ and $H_1=F$. This relatively trivial
choice will be enough to prove Theorem~\ref{thm:B}. Not every (nonamenable)
group admits nontrivial choices, but, as we show below, the free group $%
\mathbb{F}_\infty$ does.

\begin{lma}
\label{lma:propF} The free group $\mathbb{F}_\infty$ is infinitely separated.
\end{lma}

\begin{proof}
Let $\{x_n\colon n\in{\mathbb{N}}\}$ be free generators of $\mathbb{F}%
_\infty $, and let ${\mathcal{P}}\subseteq {\mathbb{N}}$ denote the set of
prime numbers. For $p\in{\mathcal{P}}$, let $H_p$ be the normal subgroup
generated by 
\begin{equation*}
\{ x_1,\ldots,x_{p-1},x_p^p,x_{p+1},\ldots\} \subseteq \mathbb{F}_\infty 
\text{.}
\end{equation*}
It is clear that $[\mathbb{F}_\infty : H_p]=p$ for all $p\in{\mathcal{P}}$,
so property (S.1) is satisfied.

We proceed to check property (S.2). For $p\in {\mathcal{P}}$, the quotient
map 
\begin{equation*}
\pi_{H_p}\colon C^*(\mathbb{F}_\infty)\to C^*({\mathbb{Z}}_p)\cong {\mathbb{C%
}}^p
\end{equation*}
can be described as follows. Identify $C^*(\mathbb{F}_\infty)$ with the full
free product $\ast_{n=1}^\infty C(S^1)$ amalgamated over $\mathbb{C}$
(identified with the algebra of scalar multiples of the identity in $C(S^1)$%
) by sending $x_n$ to the canonical generator of the $n$-th free factor $%
C(S^1)$. Given $f\in C(S^1)$ and $n\in{\mathbb{N}}$, set $f^{(n)}=1\ast
\cdots \ast f\ast 1\cdots \in \ast_{n=1}^\infty C(S^1)$, where the
nontrivial entry is in the $n$-th position. Then 
\begin{equation*}
\pi_{H_p}(f^{(n)})=%
\begin{cases}
(f(1),\ldots,f(1)), & \text{if } n\neq p \\ 
(f(1), f(e^{2\pi i/p}),\ldots,f(e^{2\pi i(p-1)/p})) & \text{if } n=p%
\end{cases}%
\end{equation*}
Now $p\in {\mathcal{P}}$ and $Q\subseteq {\mathcal{P}}$ and suppose that $%
p\notin Q$. Let $f\in C(S^1)$ be any function satisfying $f(1)=0$ and $%
f(e^{2\pi i/p})\neq 0$. Then $f^{(p)}$ belongs to $\ker(\pi_{H_{q}})$ for
all $q\in Q$, but not to $\ker(\pi_{H_p})$. Thus (S.2) is satisfied as well.
\end{proof}

From now on, we fix a discrete group $G$, a subset ${\mathcal{P}}\subseteq{%
\mathbb{N}}$ of relatively prime numbers, and a subgroup $F\leq G$ which is
separated over ${\mathcal{P}}$, as witnessed by a family $\{ H_p\} _{p\in 
\mathcal{P}}$ of subgroups of $F$.

\begin{nota}
Given $p\in{\mathcal{P}}$, we establish the following notations:

\begin{itemize}
\item we write $G_p=G/H_p$ and $F_p=F/H_p$ for the coset spaces;

\item we write $G\curvearrowright ^{\sigma _{G}^{p}}G_{p}$ and $%
F\curvearrowright ^{\sigma _{F}^{p}}F_{p}$ for the canonical translation
actions.
\end{itemize}
\end{nota}

\begin{df}
For a (possibly empty) subset $P\subseteq {\mathcal{P}}$, set 
\begin{equation*}
X^G_P=\bigsqcup\limits_{n\in{\mathbb{N}}} G\sqcup \bigsqcup_{n\in{\mathbb{N}}%
}\bigsqcup_{p\in P}G_p \ \ \mbox{ and } \ \ X^F_P=\bigsqcup\limits_{n\in{%
\mathbb{N}}} F\sqcup \bigsqcup_{n\in{\mathbb{N}}}\bigsqcup_{p\in P}F_p,
\end{equation*}
and define actions $G\curvearrowright^{\sigma^G_P} X^G_P$ and $%
F\curvearrowright^{\sigma^F_P} X^F_P$ as follows:

\begin{itemize}
\item $\sigma_P^G$ acts on each copy of $G$ via $\sigma_G$, and on each copy
of $G_p$ via $\sigma^p_G$;

\item $\sigma_P^F$ acts on each copy of $F$ via $\sigma_F$, and on each copy
of $F_p$ via $\sigma^p_F$.
\end{itemize}
\end{df}

For later use, we identify here the restriction of $\sigma_P^G$ to $F$ with
(some amplification of) $\sigma_P^F$. Recall that $G$ is countable by
assumption.

\begin{lma}
\label{lma:RestrsPtoF} Let the notation be as above, and let $P\subseteq {%
\mathcal{P}}$. Then the restriction $\sigma^G_P|_F$ of $\sigma^G_P$ to $F$
is conjugate to $\sigma_P^F$.
\end{lma}

\begin{proof}
We denote by $F\backslash G$ the right coset space of $F$ in $G$. It is
enough to show that $\sigma_G|_F$ is conjugate to $\bigsqcup\limits_{x\in
F\backslash G}\sigma_F$, and that $\sigma^p_G|_F$ is conjugate to $%
\bigsqcup\limits_{x\in F\backslash G}\sigma^p_F$ for all $p\in {\mathcal{P}}$%
. For the first one, choose a section $t\colon F\backslash G\to G$, and
define a map $f\colon G\to F\backslash G \times F$ by $f(g)=(Fg, g \cdot
t(Fg)^{-1})$ for all $g\in G$. It is clear that $f$ is a bijection, and we
claim that it intertwines $\sigma_G|_F$ and ${\mathrm{id}}_{F\backslash
G}\times \sigma|_F$. Given $k\in F$ and $g\in G$, we have 
\begin{align*}
f(\sigma_G(k)g)&=f(kg) \\
&=(Fkg, kg \cdot t(Fkg)^{-1})=(Fg, kg\cdot t(Fg)^{-1}) \\
&=({\mathrm{id}}_{F\backslash G}\times \sigma_F(k))(f(g)),
\end{align*}
as desired. The result follows, since ${\mathrm{id}}_{F\backslash G}\times
\sigma|_F$ is precisely $\bigsqcup\limits_{x\in F\backslash G}\sigma^p_F$.

The proof for $\sigma_G^p$ is completely analogous, using right coclasses.
\end{proof}

Recall the notation $u_\sigma$ from \autoref{df:UsGenBernoulli} and $\kappa^{%
\mathcal{H}}_\sigma$ from \autoref{nota:kappaSgma}.

\begin{thm}
\label{thm:sPWkCtmnt} Let $({\mathcal{H}},\eta)$ be a separable Hilbert
space with a distinguished unit vector. Let $P\subseteq {\mathcal{P}}$ be a
(possibly empty) subset, and let $\sigma_P^G$ be as in the discussion above.
Following \autoref{nota:kappaSgma}, we abbreviate $\kappa_{\sigma_P^G}^{%
\mathcal{H}}$ to simply $\kappa_{\sigma_P^G}$. Then 
\begin{equation*}
\kappa_{\sigma^G_P}|_F\subseteq u_{\sigma_P}|_F\oplus 1_G \ \ \mbox{ and } \
\ \kappa_{\sigma^G_P}^{(0)}|_F\subseteq u_{\sigma_P}|_F.
\end{equation*}
\end{thm}

\begin{proof}
We begin by providing an alternative description of the restriction $%
\kappa_{\sigma_P}|_F$ of $\kappa_{\sigma_P}$ to $F$. Find an orthonormal
basis $\{\eta_n\colon n\in {\mathbb{N}}\}$ of ${\mathcal{H}}$ with $%
\eta_0=\eta$, and set 
\begin{equation*}
\mathcal{F}=\{\xi \colon X_p\to {\mathbb{N}}\colon \xi(x)=0 
\mbox{ for all
but finitely many } x\in X_P\}.
\end{equation*}
By \autoref{lma:BasisHs}, $\mathcal{F}$ is an orthonormal basis for $%
\bigotimes\limits_{x\in X_P}{\mathcal{H}}$. Moreover, the unitary
representation $\kappa_{\sigma^G_P}|_F$ restricts to an action of $F$ on $%
\mathcal{F}$, which is given by 
\begin{equation*}
(\kappa_{\sigma^G_P}(g)(\xi))(x)=\xi(g^{-1}\cdot x)
\end{equation*}
for $g\in F$, for $\xi\in \mathcal{F}$ and for $x\in X$. We denote by $%
\xi_0\in \mathcal{F}$ the function which is identically zero, and write $%
\mathcal{F}_0$ for $\mathcal{F}\setminus\{\xi_0\}$. Then $\mathcal{F}_0$ is
an orthonormal basis for the orthogonal complement of $\eta$, and it is also
invariant under $\kappa_{\sigma^G_P}|_F$. (In fact, the restriction of $%
\kappa_{\sigma^G_P}|_F$ to $\overline{\mathrm{span}}(\mathcal{F}_0)$ is $%
\kappa_{\sigma^G_P}^{(0)}|_F$.)

Denote by $\mathcal{G}$ and $\mathcal{G}_0$ the $F$-orbit spaces of $%
\mathcal{F}$ and $\mathcal{F}_0$, respectively. For $\xi\in\mathcal{F}$, we
write $[\xi]\in \mathcal{G}$ for its equivalence class, and $\mathrm{Stab}%
_F(\xi)$ for the stabilizer subgroup of $F$. It follows that 
\begin{equation*}
\kappa_{\sigma^G_P}|_F\cong \bigoplus_{[\xi]\in \mathcal{G}}\lambda_{F/%
\mathrm{Stab}_F(\xi)} \ \ \mbox{ and } \ \ \kappa^{(0)}_{\sigma^G_P}|_F\cong
\bigoplus_{[\xi]\in \mathcal{G}_0}\lambda_{F/\mathrm{Stab}_F(\xi)}.
\end{equation*}
It is clear that $\lambda_{F/\mathrm{Stab}_F(\xi_0)}$ is the trivial
representation of $G$.

\textbf{Claim:} \emph{Fix $\xi\in\mathcal{F}_0$. Then $\lambda_{F/\mathrm{%
Stab}_F(\xi)}$ is unitarily contained in $u_{\sigma_P}|_F$}.\newline
By part~(1) of \autoref{lma:UnionSn} and \autoref{lma:RestrsPtoF}, the
representation $u_{\sigma_P}|_F$ is unitarily equivalent to $%
\bigoplus\limits_{n\in{\mathbb{N}}}\lambda_F\oplus \bigoplus\limits_{n\in{%
\mathbb{N}}}\bigoplus\limits_{p\in P}\lambda_{F_p}$. We divide the proof
into two cases.

Assume first that the support $\mathrm{supp}(\xi)$ of $\xi$ meets one of the
copies of $G$ in $X_P$. If $k\in \mathrm{Stab}_F(\xi)$, then $k\cdot {%
\mathrm{supp}}(\xi)={\mathrm{supp}}(\xi)$. By restricting this equality to
the copy of $G$ which ${\mathrm{supp}}(\xi)$ intersects, we deduce that only
finitely many group elements $k\in F$ satisfy this identity, so $\mathrm{Stab%
}_F(\xi)$ is finite. (In fact, in this case even $\mathrm{Stab}_G(\xi)$ is
finite.) Thus $\lambda_{F/\mathrm{Stab}_F(\xi)}\subseteq \lambda_F$ by %
\autoref{lma:QregRepContained}, so $\lambda_{F/\mathrm{Stab}%
_F(\xi)}\subseteq u_{\sigma^G_P}|_F$.

Assume now that the support of $\xi$ does not meet any of the copies of $G$
in $X_P$. Let $(q_1,\ldots,q_m)$ be a minimal tuple of elements of $P$
(potentially with repetitions) such that ${\mathrm{supp}}(\xi)\subseteq
\bigsqcup\limits_{j=1}^m G_{q_j}$. Denote by $p_1,\ldots, p_n$ the distinct
values of $q_1,\ldots,q_m$, and set $H=H_{p_1}\cap \cdots\cap H_{p_n}$. Then
Stab$_F(\xi)\supseteq H$. In particular, Stab$_F(\xi)$ has finite index in $%
F $. Using \autoref{lma:QregRepContained} at the first step, and part~(2) of %
\autoref{lma:IndexMultipl} at the second, we conclude that 
\begin{equation*}
\lambda_{F/\mathrm{Stab}_F(\xi)}\subseteq \lambda_{F/H}\cong
\lambda_{F_{p_1}}\oplus \cdots \oplus \lambda_{F_{p_n}}\subseteq
u_{\sigma^G_P}|_F.
\end{equation*}
The claim is proved. It follows that 
\begin{equation*}
\kappa_{\sigma^G_P}^{(0)}|_F\cong \bigoplus_{[\xi]\in \mathcal{G}%
_0}\lambda_{F/\mathrm{Stab}_F(\xi)}\subseteq u_{\sigma^G_P}|_F,
\end{equation*}
and thus also $\kappa_{\sigma^G_P}|_F\subseteq u_{\sigma^G_P}|_F\oplus 1_G$,
as desired.
\end{proof}

The following theorem connects weak equivalence of certain representations
to cocycle conjugacy of the associated generalized Bernoulli shifts (see %
\autoref{df:UsGenBernoulli}). In particular, it implies that for a \emph{%
nonamenable} countable discrete group $G$, the generalized Bernoulli actions
obtained from different subsets of ${\mathcal{P}}$ are not cocycle
equivalent.

\begin{thm}
\label{thm:MainTool} Let $D$ be a tracial unital C*-algebra or a tracial
von Neumann algebra, let $G$ be a discrete group containing a nonamenable
subgroup $F$ which is separated over some set ${\mathcal{P}}\subseteq {%
\mathbb{N}}$, and adopt the notation from before \autoref{thm:sPWkCtmnt}.
For $P,Q\subseteq {\mathcal{P}}$, the following are equivalent:

\begin{enumerate}
\item $P=Q$.

\item $u_{\sigma^G_P}\cong u_{\sigma^G_Q}$.

\item $u_{\sigma^G_P}|_F\sim_{\mathrm{Z}} u_{\sigma^G_Q}|_F$

\item $\beta_{\sigma^G_P,D}$ is cocycle conjugate to $\beta_{\sigma^G_Q,D}$.

\item For every (normal) tracial state $\tau$ on $D$ for which $\overline{D}%
^{\tau}$ is separable and not one-dimensional, the action $\beta_{\sigma^G_P,%
\overline{D}^\tau}$ is cocycle conjugate to $\beta_{\sigma^G_Q,\overline{D}%
^\tau}$.
\end{enumerate}
\end{thm}

\begin{proof}
The implications (1) $\Rightarrow$ (2) $\Rightarrow$ (3), and (1) $%
\Rightarrow$ (4) $\Rightarrow$ (5) are immediate. It therefore suffices to
prove that (5) implies (3), and that (3) implies (1).

(5) $\Rightarrow $ (3). Assume that $\beta _{\sigma _{P}^{G},\overline{D}}$
is cocycle conjugate to $\beta _{\sigma _{Q}^{G},\overline{D}}$. 
We apply \autoref{prop:sRpWkCtmnt} at the first step with $\alpha =\beta
_{\sigma _{P}^{G},\overline{D}}$, $\beta =\beta _{\sigma _{Q}^{G},\overline{D%
}}$ and $\sigma =\sigma _{P}^{G}$ (and restricting to the subgroup $F$), and %
\autoref{thm:sPWkCtmnt} at the second step, to deduce that 
\begin{equation*}
u_{\sigma _{P}^{G}}|_{F}\prec _{\mathrm{Z}}\kappa _{\sigma _{Q}^{G},%
\overline{D}}^{(0)}|_{F}\subseteq u_{\sigma _{Q}^{G}}|_{F}.
\end{equation*}%
In view of \autoref{prop:sRpCocConj}, by reversing the roles of $P$
and $Q$, we conclude that $u_{\sigma _{P}^{G}}|_{F}\sim _{\mathrm{Z}%
}u_{\sigma _{Q}^{G}}|_{F}$, as desired.

(3) $\Rightarrow$ (1). Let $p\in P$, and suppose that 
\begin{equation*}
\lambda_{F_p}\prec_{\mathrm{Z}} u_{\sigma^G_Q}|_F\cong \bigoplus_{n\in{%
\mathbb{N}}}\lambda_F\oplus \bigoplus_{n\in{\mathbb{N}}}\bigoplus_{q\in
Q}\lambda_{F_q}.
\end{equation*}
Let $\mu$ be an irreducible subrepresentation of $\lambda_{F_p}$. Then it
follows from \autoref{prop:IrredRepWkCont} that either $\mu\prec_{\mathrm{Z}%
} \bigoplus\limits_{n\in{\mathbb{N}}} \lambda_F$ or $\mu\prec_{\mathrm{Z}}
\bigoplus\limits_{n\in{\mathbb{N}}} \bigoplus\limits_{q\in Q}\lambda_{F_q}$.
Since clearly $1_F \prec_{\mathrm{Z}}\mu$, the first case would imply that $%
F $ is amenable, which contradicts our assumptions. Thus, we must have $%
\mu\prec_{\mathrm{Z}} \bigoplus\limits_{n\in{\mathbb{N}}}
\bigoplus\limits_{q\in Q}\lambda_{F_q}$. Since this applies to every
irreducible subrepresentation of $\lambda_p$, and $\lambda_p$ is unitarily
conjugate to their sum, it follows that $\lambda_{F_p}\prec_{\mathrm{Z}}
\bigoplus\limits_{n\in{\mathbb{N}}} \bigoplus\limits_{q\in Q}\lambda_{F_q}$.

Considering the induced representations of the full group C*-algebra $C^*(F)$
(using the notation from \autoref{df:separated}), and taking kernels, we
deduce that 
\begin{equation*}
\bigcap\limits_{q\in Q} \ker(\pi_{H_q})\subseteq \ker(\pi_{H_p}).
\end{equation*}
By the properties of the family $\{H_p\colon p\in\mathcal{P}\}$
(specifically, by (S.2)), we have $p\in Q$. Since $p\in P$ was arbitrary,
this shows that $P\subseteq Q$, as desired.
\end{proof}

In order to deal with $\mathcal{Z}$-stable C*-algebras, or with McDuff
factors, we will also need the following variation of \autoref{thm:MainTool}.

\begin{thm}
\label{thm:MainTool2} Let $D$ and $A$ be both tracial C*-algebras or tracial
von Neumann algebras, with $D$ unital, let $G$ be a nonamenable countable
discrete group, and adopt the notation from before \autoref{thm:sPWkCtmnt}.
For $P,Q\subseteq {\mathcal{P}}$, the following are equivalent:

\begin{enumerate}
\item $P\cup\{1\}=Q\cup \{1\}$.

\item ${\mathrm{id}}_A\otimes \beta_{\sigma^G_P,D}$ is cocycle conjugate to $%
{\mathrm{id}}_A\otimes \beta_{\sigma^G_Q,D}$.

\item For every (normal) tracial states $\tau \in \mathcal{S}(D)$ and $\tau
_{A}\in \mathcal{S}(A)$ for which $\overline{D}^{\tau }$ and $\overline{A}%
^{\tau _{A}}$ are separable and not one-dimensional, the action ${\mathrm{id}%
}_{\overline{A}^{\tau _{A}}}\otimes \beta _{\sigma _{P}^{G},\overline{D}}$
is cocycle conjugate to ${\mathrm{id}}_{\overline{A}^{\tau _{A}}}\otimes
\beta _{\sigma _{Q}^{G},\overline{D}}$.
\end{enumerate}
\end{thm}

\begin{proof}
The proof is almost identical to that of \autoref{thm:MainTool}, using the
fact that the Koopman representation of ${\mathrm{id}}_A\otimes
\beta_{\sigma^G_P,D}$ is conjugate to that of $\beta_{\sigma^G_{P\cup\{1%
\}},D}$.
\end{proof}

\section{Main results}

In this section, we prove Theorem~\ref{thm:B} and Theorem~\ref{thm:C} from
the introduction, which make significant contributions to part~(2) of
Conjecture~\ref{cnjintro}.

\begin{df}
Let $D$ be a tracial C*-algebra, and let $\theta \in {\mathrm{Aut}}(D)$. We
say that $\theta $ is \emph{strongly outer} if for every $\tau
\in T(D)$ satisfying $\tau \circ \theta =\tau $, the weak extension $%
\overline{\theta }^{\tau }\in {\mathrm{Aut}}(\overline{D}^{\tau })$ is
outer. An action $\alpha \colon G\rightarrow {\mathrm{Aut}}(D)$ of a
discrete group $G$ is said to be \emph{strongly outer} if $\alpha
_{g}$ is a strongly outer automorphism of $D$ for every $g\in G\setminus
\{1\}$.
\end{df}

For later use, we record here the following standard fact.

\begin{prop}
\label{prop:GenBerShiftOuter} Let $G$ be an infinite, countable group, and
let $D$ be a tracial unital C*-algebra. Then the Bernoulli shift $\beta
_{D}\colon G\rightarrow {\mathrm{Aut}}(\bigotimes_{g\in G}D)$ is strongly
outer. More generally, if $A$ is any C*-algebra and $\alpha \colon
G\rightarrow {\mathrm{Aut}}(A)$ is any action, then $\beta _{D}\otimes
\alpha $ is strongly outer.
\end{prop}

\begin{thm}
\label{thm:TwoActs} Let $G$ be a nonamenable countable discrete group, and
let $D$ be a tracial C*-algebra or a tracial von Neumann algebra. Then the
Bernoulli shift $\beta _{D}\colon G\rightarrow {\mathrm{Aut}}%
(\bigotimes_{g\in G}D)$ does not tensorially absorb the trivial action on
any unital $C^{\ast }$-algebra\ (or von Neumann algebra) tensorially. In
particular, $\beta _{D}$ and $\beta _{D}\otimes {\mathrm{id}}_{D}$ are two
non-cocycle conjugate, strongly outer actions of $G$ on $\bigotimes_{n\in {%
\mathbb{N}}}D$.
\end{thm}

\begin{proof}
Take $P=\emptyset$ and $Q={1}$, so that, in the notation from before \autoref%
{prop:sRpWkCtmnt}, one has $\sigma^G_P=\lambda_G^{\infty}$ and $%
\sigma^G_Q=\lambda_G^{\infty}\oplus 1_G^{\infty}$. Choose a (normal) state $%
\phi\in \mathcal{S}(D)$ for which $\overline{D}^{\phi}$ is separable and not
one-dimensional. It then follows from the equivalence between parts~(1)
and~(5) in~\autoref{thm:MainTool} that $\beta_{\sigma^G_P,D}=\beta_D$ is not
cocycle conjugate to $\beta_{\sigma^G_Q,D}=\beta_D\otimes{\mathrm{id}}_D$.
Finally, these actions are strongly outer by \autoref{prop:GenBerShiftOuter}.
\end{proof}

We have arrived at Theorem~B.

\begin{thm}
\label{thm:CharactAmenBerShift} Let $G$ be a countable group, and let ${%
\mathcal{D}}$ be a finite strongly self-absorbing C*-algebra. Then the
following are equivalent:

\begin{enumerate}
\item $G$ is amenable;

\item The Bernoulli shift $\beta_{\mathcal{D}}\colon G\to{\mathrm{Aut}}%
(\bigotimes\limits_{g\in G}\mathcal{D})$ is cocycle conjugate to $\beta_{%
\mathcal{D}}\otimes{\mathrm{id}}_{\mathcal{Z}}$.
\end{enumerate}
\end{thm}

\begin{proof}
That (1) implies (2) is a consequence of Corollary~4.8 in~\cite%
{GarHir_strongly_2017} (and it can also be deduced from the proof of
Theorem~1.1 in~\cite{Sat_actions_2016}). The converse follows immediately
from \autoref{thm:TwoActs}, since $\bigotimes_{n\in {\mathbb{N}}}\mathcal{D}%
\cong \mathcal{D}$.
\end{proof}

The theorem above complements the results in \cite{Sat_actions_2016} and 
\cite{GarHir_strongly_2017} quite nicely: while every strongly outer action
of an amenable group on a tracial strongly self-absorbing C*-algebra
absorbs the identity on $\mathcal{Z}$ tensorially, this result fails for
every nonamenable group. In particular, we deduce a weak version of part~(2)
of Conjecture~\ref{cnjintro}: any nonamenable group admits at least two
non-cycle conjugate, strongly outer actions on $\mathcal{D}$, namely, the
Bernoulli shift $\beta_{\mathcal{D}}$ and $\beta_{\mathcal{D}}\otimes{%
\mathrm{id}}_{\mathcal{Z}}$.

Our strongest result, which will imply Theorem~\ref{thm:C} and Theorem~\ref%
{thm:D}, assumes the existence of an infinitely separated subgroup, in the
sense of \autoref{df:separated}.

\begin{thm}
\label{thm:UnctblyActs} Let $G$ be a countable nonamenable group containing
an infinitely separated subgroup, and let $D$ and $A$ be both tracial
C*-algebras or tracial von Neumann algebras, with $D$ unital. Then there
exist uncountably many pairwise non-cocycle conjugate, strongly outer
actions of $G$ on $A\otimes \bigotimes_{n\in {\mathbb{N}}}D$.
\end{thm}

\begin{proof}
Assume first that both $D$ and $A$ are C*-algebras. Let $F$ be a subgroup of 
$G$ and let ${\mathcal{P}}$ be an infinite subset of ${\mathbb{N}}$ over
which $F$ is separated. For a subset $P\subseteq {\mathcal{P}}$ containing
1, consider the action $G\curvearrowright^{\sigma^G_P}X_P$ defined before %
\autoref{thm:sPWkCtmnt}. Let $\tau\in \mathcal{S}(D)$ be a tracial state for
which $\overline{D}^{\tau}$ is separable and not one-dimensional. By the
equivalence between~(1) and~(2) in~\autoref{thm:MainTool2}, the family $\{{%
\mathrm{id}}_A\otimes \beta_{\sigma^G_P,D}\colon 1\in P\subseteq {\mathcal{P}%
}\}$ consists of pairwise non-cocycle conjugate actions of $G$ on $A\otimes
\bigotimes_{n\in{\mathbb{N}}}D$. Since ${\mathcal{P}}$ is infinite, this
family is uncountable. Finally, these actions are strongly outer by \autoref%
{prop:GenBerShiftOuter}, so the proof is finished.

The same argument applies in the case of tracial von Neumann algebras, using
the equivalence between~(1) and~(3) in~\autoref{thm:MainTool2}.
\end{proof}

As an immediate consequence, we obtain the following, which implies Theorem~%
\ref{thm:C}.

\begin{cor}
\label{cor:UnctblyActsSSA} Let $G$ be a nonamenable group containing an
infinitely separated subgroup, and let $A$ be a tracial $\mathcal{Z}$-stable
C*-algebra. Then there exist uncountably many pairwise non-cocycle
conjugate, strongly outer actions of $G$ on $A$, which are pointwise
asymptotically inner.
\end{cor}

\begin{proof}
This follows immediately from \autoref{thm:UnctblyActs}, by taking $D=%
\mathcal{Z}$. The statement about pointwise asymptotic innerness follows
from the fact that any automorphism of $\mathcal{Z}$ is automatically
asymptotically inner.
\end{proof}

By \autoref{lma:propF}, the corollary above applies to any group containing
a free group, and hence implies Theorem~\ref{thm:C}. Also, by Winter's
theorem \cite{Win_strongly_2011}, the result above applies to any strongly
self-absorbing C*-algebra, thus proving part~(2) of Conjecture~\ref{cnjintro}
for groups containing a free group. This can be regarded as a the
noncommutative analog of Ioana's celebrated result \cite{Ioa_orbit_2011}.

A similar result holds for McDuff von Neumann algebras, thus strengthening a
result of Brothier-Vaes (Theorem~B in~\cite{BroVae_families_2015}) for
groups containing a free group. Even in the case of $\mathcal{R}$, our
methods offer a simpler and shorter proof of their theorem, which avoids
Popa's very advanced theory of spectral gap
rigidity; see \cite{Pop_strong_2003} and \cite{Pop_some_2006}.

\begin{cor}
\label{cor:UnctblyActsMcDuff} Let $G$ be a nonamenable group containing an
infinitely separated subgroup, and let $M$ be a tracial McDuff von Neumann
algebra. Then there exist uncountably many pairwise non-cocycle conjugate,
outer actions of $G$ on $M$.
\end{cor}

\begin{proof}
This follows immediately from \autoref{thm:UnctblyActs}, by taking $D=%
\mathcal{R}$.
\end{proof}

In the measurable setting, Epstein \cite{Eps_orbit_2007} combined Ioana's
result from \cite{Ioa_orbit_2011} with Gaboriau-Lyons' solution \cite%
{GabLyo_measurable_2009} to the von Neumann problem, to show that \emph{any}
nonamenable group admits a continuum of non-orbit equivalent free, ergodic
actions. In order to prove part~(2) of Conjecture~\ref{cnjintro} for all
nonamenable groups, one could attempt a similar approach of inducing actions
from $\mathbb{F}_2$ to any amenable group. For this approach to work,
however, one would need a noncommutative analog of the result of
Gaboriau-Lyons. This suggests the following interesting problem:

\begin{pbm}
Is there an analog of Gaboriau-Lyon's measurable solution to the von Neumann
problem in the context of strongly outer actions on (finite) strongly
self-absorbing C*-algebras? And for outer actions on the hyperfinite II$_1$%
-factor?
\end{pbm}


\end{document}